\documentclass[11pt,reqno,oneside]{amsart}
\usepackage[utf8]{inputenc}
\usepackage[english]{babel}
\usepackage[babel]{csquotes}

\usepackage{amsmath}
\usepackage{amsfonts}
\usepackage{amssymb}
\usepackage{amsthm} 
\usepackage{dsfont} 
\usepackage[mathscr]{euscript} 
\usepackage{eurosym}
\usepackage{bbm} 
\usepackage[a4paper]{geometry}
\usepackage{comment} 
\usepackage{paralist} 

\usepackage{graphicx}
\usepackage{xkeyval}
\usepackage{pstricks}
\usepackage{hyperref}
\usepackage{longtable}
\usepackage{tikz}

\hypersetup{
urlcolor=black, 
  menucolor=black, 
  citecolor=black, 
  anchorcolor=black, 
  filecolor=black, 
  linkcolor=black, 
  colorlinks=true,
}
\usepackage{comment}

\usetikzlibrary{matrix}

\newcommand{\Prob}{\mathds{P}}

\newcommand{\E}{\mathds{E}}
\newcommand{\V}{\mathds{V}}
\newcommand{\C}{\mathbb{C}}
\newcommand{\R}{\mathbb{R}}

\newcommand{\Z}{\mathbb{Z}}
\newcommand{\N}{\mathbb{N}}

\newcommand{\abs}[1]{|{#1}|} 
\newcommand{\bigabs}[1]{\left|{#1}\right|} 
\newcommand{\one}{\mathds{1}}

\newcommand{\Acal}{\mathcal{A}}

\newcommand{\Pcal}{\mathcal{P}}

\newcommand{\Ncal}{\mathcal{N}}

\newcommand{\Lcal}{\mathcal{L}}

\newcommand*{\defeq}{\mathrel{\vcenter{\baselineskip0.5ex \lineskiplimit0pt
                     \hbox{\scriptsize.}\hbox{\scriptsize.}}}%
                     =}

\newcommand{\supnorm}[1]{\|#1\|_{\infty}}
\newcommand{\onenorm}[1]{\|#1\|_{1}}

\newcommand{\norm}[1]{\|#1\|}

\newcommand{\ubar}[1]{\underline{#1}}
\newcommand{\oneto}[1]{[{#1}]}

\DeclareMathOperator{\diam}{diam}

\DeclareMathOperator{\dist}{dist}

\theoremstyle{plain}
\newtheorem{lemma}{Lemma}
\newtheorem{theorem}[lemma]{Theorem}

\newtheorem{corollary}[lemma]{Corollary}
\theoremstyle{definition}

\newtheorem{example}[lemma]{Example}
\newtheorem{remark}[lemma]{Remark}
\theoremstyle{remark}

\begin{document}
\title[Weak Dependence by Moments]
{The Central Limit Theorem for Weakly Dependent Random Variables by the Moment Method}
\author[Michael Fleermann]{Michael Fleermann}
\author[Werner Kirsch]{Werner Kirsch}
\begin{abstract}
In this paper, we derive a central limit theorem for collections of weakly correlated random variables indexed by discrete metric spaces, where the correlation decays in the distance of the indices. The correlation structure we study depends solely on the separability of mixed moments. Our investigation yields a new proof for the CLT for $\alpha$-mixing random variables, but also non-$\alpha$-mixing random variables fit within our framework, such as MA($\infty$) processes. In particular, our results can be applied to ARMA($p,q$) process with independent white noise.
\end{abstract}
\maketitle

\section{Introduction}
\label{sec:Intro}
The study of limit theorems for weakly dependent random variables has a long history of extremely extensive and fruitful research. For an overview of many facets of this theory we refer the reader to the monograph \cite{Dedecker:2007}. A particularly prominent concept to measure dependence between random variables consists of constructing so-called "strong mixing conditions" for the deviations of probabilities of events in the $\sigma$-algebras which are generated by the random variables, see \cite{Bradley:2005} for details. This research thread was initiated by the paper \cite{Rosenblatt:1956}, in which the concept of "strong mixing" or "$\alpha$-mixing" was introduced to derive a central limit theorem for stationary sequences of random variables. Generalizations and extensions of these findings are included, for example, in the works of \cite{Ibragimov:1962}, \cite{Bolthausen:1982} and \cite{Doukhan:Louhichi:1999}. What all these studies have in common is that their limit theorems are mainly formulated for stationary sequences indexed by $\N$, $\Z$ or $\Z^d$.
The purpose of this paper is to provide another model of weak dependence which solely depends on the separability of mixed moments of random variables, see our condition (M4) below. Naturally -- since the dependence condition is formulated in the language of moments -- we employ the method of moments to prove our main result, Theorem~\ref{thm:CLT}. In contrast to the previous investigations mentioned above, we do not require stationarity of the variables involved, and our index sets are general metric spaces. It turns out that multiple models such as $\alpha$-mixing sequences and MA($\infty$) random fields fit our framework.

This paper is organized as follows: In Section~\ref{sec:SetupResults} we present the setup and the main result of this paper, Theorem~\ref{thm:CLT}. In Section~\ref{sec:MAinfty} we introduce the first class of stochastic processes that matches our setup, the so-called MA($\infty$) processes, which are moving averages of infinite time horizon. In Corollaries~\ref{cor:MAInftyN} and~\ref{cor:MAInftyZd} we establish two exemplary applications of Theorem~\ref{thm:CLT} to MA($\infty$) processes. In Section~\ref{sec:Mixing}, we show that $\alpha$-mixing random variables also fit within our framework and derive a corresponding central limit theorem in Corollary~\ref{cor:Mixing}. Lastly, in Section~\ref{sec:proof} we carry out the proof of Theorem~\ref{thm:CLT}.

\subsection*{Acknowledgements} The first author thanks Felix Spangenberg for valuable literature hints on stochastic time series.

\section{Setup and Main Results}
\label{sec:SetupResults}
We assume that for all $n\in\N$, $(X^{(n)}_t)_{t\in T_n}$ is a collection of real-valued random variables, where $(T_n,d_n)$ is a finite metric space with $\N_0$-valued metric $d_n$. As an example, we think of $T_n\subseteq \Z^m$, $d_n(z,z')=\max_{i\in\oneto{m}}\abs{z_i-z'_i}$, where for all $m\in\N$ we write $\oneto{m}\defeq\{1,\ldots,m\}$. For notational convenience, we write 
$X_t=X^{(n)}_t$ and $d=d_n$ when $n$ is clear from the context. 

We impose the following moment assumptions on these random variables:
\begin{enumerate}[(M1)]
\item All $X_t$ are centered.
\item $\frac{1}{\abs{T_n}}\sum_{t\in T_n}\E X_t^2$ converges as $n\to\infty$.
\item $\forall\,k\in\N:\,\exists\, M_k\geq 0:\,\forall\, n\in\N:\,\forall\,t\in T_n:\ \E\abs{X_t}^k\leq M_k$.
\end{enumerate}

In addition to these moment assumptions, we make an additional assumption on the separability of mixed moments. To formulate this condition, we introduce some concepts and notation we will use for the remainder of this paper. A family $S\subseteq T_n$ is called \emph{$a$-connected}, where $a\in\N_0$, if for all $s,s'\in S$ we find finitely many elements $s_1,\ldots,s_{\ell}\in S$ such that with $s_0\defeq s$ and $s_{\ell+1}\defeq s'$ it holds $d(s_{i-1},s_{i})\leq a$ for all $i\in\oneto{\ell+1}$. Further, if $I$ is an index set and for all $i\in I$, $t_i\in T_n$, then we write $t_I \defeq (t_i\,|\,i\in I)$ as a family (or vector) of elements in $T_n$. Now given a vector $\ubar{t}\defeq (t_1,\ldots,t_k) \in T_n^k$ and an $a\in\N_0$, let $\pi(\ubar{t} \, | \, a)\in\Pcal(k)\defeq\{\pi\,|\, \pi \text{ partition of } \oneto{k}\}$ be the unique partition such that for all $B\in\pi(\ubar{t} \, | \, a)$, $t_B$ is $a$-connected in $T_n$ and $d(t_B,t_{B'})> a$ for all $B\neq B'\in\pi(\ubar{t} \, | \, a)$. Thus, for any tuple $\ubar{t}$, the partition $\pi(\ubar{t} \, | \, a)$ allows us to decompose $\ubar{t}$ into its $a$-connected parts. The moment separation condition we impose is 
\begin{equation}
\label{eq:momentseperation}
	\bigabs{\E X_{t_1}\cdots X_{t_k} - \prod_{B\in\pi(\ubar{t} \, | \, a)}\E X_{t_B}} \leq C_k\gamma (a),
	\tag{M4}
\end{equation}
where $C_k$ is a constant depending on $k$ (and the underlying process) but not on $n$, and $\gamma:\N_0 \to \R_+$ is a function such that $\gamma(a)\to 0$ as $a\to\infty$. 
We also impose the following conditions on the spaces $T_n$ and $\gamma$, where $B_a(t)\subseteq T_n$ refers to the ball of radius $a$ around $t$ in $T_n$:
\begin{enumerate}[(S1)]
\item For all $k\in\N$: $\abs{T_n}^k\gamma(n)\to 0$ as $n\to\infty$.	
\item $\exists\, C,c > 0:\, \forall\,n,a\in\N: b_n(a)\leq C\abs{T_a}^c$, where $b_n(a)\defeq \max_{t\in T_n} \abs{B_a(t)}$
\item $\forall\,n\in\N: \abs{T_n}\geq n$, and $\diam(T_{n+1})\geq \diam(T_n)$.
\item $\forall\,\ell\in\N: \exists\, (a_n)_n$ with $a_n\nearrow\infty$, $a_n\in\{1,\ldots,n\}$, such that for finally all $n\in\N$: $\abs{T_{a_n}}^{\ell}/\abs{T_n} \in(c',C')$ for some constants $0<c'<C'$ (possibly depending on $\ell$).	
\end{enumerate}
Condition (S1) requires a decay of $\gamma$ which is superpolynomial in $\abs{T_n}$. Condition (S2) is an evolvement condition which states that the cardinality of $T_a$ should evolve similarly to the cardinality of balls of radius $a$. Condition (S3) ensures a minimal growth of $\abs{T_n}$ and therefore is also an evolvement condition. Condition (S4) states that the growth of $T_n$ can be recuperated by the growth of arbitrary $\ell$-th powers of former spaces $T_m$, $m\leq n$. In total, the conditions (S1) through (S4) pertain mostly to the cardinality of the spaces $T_n$, but not so much to their topological structure, as conditions formulated in \cite{Bolthausen:1982}.
Lastly, we want to point out that although the conditions (S1) through (S4) seem cumbersome, they are actually very easily verified in applications, see the proofs of the Corollaries~\ref{cor:MAInftyN},~\ref{cor:MAInftyZd} and~\ref{cor:Mixing}.

\begin{theorem}
\label{thm:CLT}
Under the conditions (M1), (M2), (M3), (M4), (S1), (S2), (S3), and (S4) we find that
\begin{equation}
\label{eq:sigmalimit}
\sigma^2\defeq \lim_{n\to\infty}\frac{1}{\abs{T_n}}\sum_{s,t\in T_n}\E X_s X_t	\qquad \text{exists in $\R_+$}
\end{equation}
and
\begin{equation}
\label{eq:CLT}
\frac{1}{\sqrt{\abs{T_n}}}\sum_{t\in T_n} X_t	\xrightarrow[n\to\infty]{} \Ncal(0,\sigma^2) \qquad \text{weakly,}
\end{equation}
where  $\Ncal(0,\sigma^2)\defeq \delta_0$ if $\sigma^2 = 0$.
\end{theorem}

Before we begin with the proof in Section~\ref{sec:proof}, we provide some examples in the following two sections.

\section{MA($\infty$)-processes}
\label{sec:MAinfty}
In this example, we analyze the behavior of general $MA(\infty$) processes and show how they fit into our setup. We assume $(Y_t)_t$ to be a sequence of independent random variables with $\E Y_t =0$ , $\V Y_t=1$ and $\E\abs{Y_t}^k\leq M'_k$ for all $k\in\N$ and all indices $t\in T$, where $(T,d)$ is a metric space, $(T,+)$ is an abelian group, and the metric $d$ is translation invariant.  As an example, consider $(\Z^m,+,d)$ with $d(x,z) = \max_{i\in m}\abs{x_i-y_i}$. We define for all $t\in T$
\begin{equation}
\label{eq:MAdef}
X_t \defeq \sum_{s\in T}c_s Y_{t-s} = \sum_{s\in T} c_{t-s} Y_s,
\end{equation}
where $(c_s)_s$ is a sequence in $\ell^1(T)$ and call $(X_t)$ an MA($\infty$)-process.
In contrast to many treatments of these processes, we do not require $(Y_t)$ to be i.i.d., so that the process $X_t$ need not be stationary. Also, in the literature MA($\infty$)-processes as we defined them are sometimes called \emph{linear processes}, cf. \cite{Brockwell:Davis:1991} and \cite{Brockwell:Davis:2016}. There, it is also illustrated that in fact many different times series models possess an MA($\infty$) representation.

In \eqref{eq:MAdef} it is straightforward to see that $X_t$ can be regarded as an almost sure limit, but also as $\Lcal_p$-limit for all $p\geq 1$. In particular, this holds true for products of the variables $X_t$, and expectations of these expressions may be evaluated summand-wise. To be more precise, if $t_1,\ldots,t_k\in T$, then we obtain the following calculation, which we explain step-by-step below:
\begin{align*}
&\E X_{t_1}\cdots X_{t_k} 
\ = \ \sum_{\ubar{s}\in T^k} c_{t_1-s_1}\cdots c_{t_k-s_k} \E Y_{s_1}\cdots Y_{s_k}\\
&= \sum_{\kappa\in\Pcal(k)} \sum_{\substack{\ubar{s}\in T^k \\ \ker \ubar{s}=\kappa }} c_{\ubar{t}-
\ubar{s}} \E Y_{\ubar{s}}\ =\ \sum_{\kappa\in\Pcal(k)_{\geq 2}} \sum_{\substack{\ubar{s}\in T^k \\ \ker \ubar{s}=\kappa }} c_{\ubar{t}-
\ubar{s}} \E Y_{\ubar{s}}\\
&=\sum_{\kappa\in\Pcal(k)_{\geq 2}} \sum_{\substack{\ubar{s}\in T^k \\ \ker \ubar{s}=\kappa }}  \prod_{K\in\kappa}c_{t_K-
s_K}\E Y_{s_K}\ =\ \sum_{\kappa\in\Pcal(k)_{\geq 2}} \prod_{K\in\kappa}\sum_{s^K\in T} c_{t_K-s^K}\cdot \E Y_{s^K}^{\abs{K}}.
\end{align*}
In the first step we used the definition in \eqref{eq:MAdef} and the discussion below the definition. In the second step, $\ker \ubar{s}$ refers to the \emph{kernel} of the tuple $\ubar{s}$, which is the unique partition on $\oneto{k}$ with $i\sim j \Leftrightarrow s_i=s_j$. In other words, in the second step we sort the tuples according to which places in the tuple contain equal or different entries. In the third step, $\Pcal(k)_{\geq 2}$ refers to the partitions with blocks of size $2$ or larger. We may restrict ourselves to these partitions, since the variables $Y_i$ are independent and centered, so that $\E Y_{\ubar{s}}$ vanishes whenever one of the indices in $\ubar{s}$ appears only once. In the fourth step, we used the independence of the variables $Y_i$. In the fifth step, we use that for all $K\in\kappa$ and $i,j\in K$, $s_i=s_j$. Here, $t_K$ is the vector $(t_i, i\in K)$, whereas $s^K$ is just a single element. The difference $t_K-s^K$ is then to be interpreted componentwise as $(t_i - s^K, i\in K)$. In particular, it immediately becomes clear from the above calculation that for any $t,t'\in T$ and $k\in\N$:
\begin{align}
&\E X_t = 0, \qquad \E X_t^2 = \sum_{s\in T} \abs{c_s}^2, \qquad \E X_tX_{t'} = \sum_{s \in T} c_{t-s}c_{t'-s}, \quad\text{ and}\notag\\
&\E\abs{X_t}^k\ \leq \sum_{\kappa\in\Pcal(k)_{\geq 2}} \prod_{K\in\kappa}\sum_{s^K\in T} \abs{c_{t-s^K}}^{\abs{K}}\cdot M'_{\abs{K}} \label{eq:MAinftyFirstMoments} \\
&\qquad\quad= \ \sum_{\kappa\in\Pcal(k)_{\geq 2}}\left(\prod_{K\in\kappa}M_{\abs{K}}'\right)\left(\prod_{K\in\kappa}\sum_{s^K\in T} \abs{c_{s^K}}^{\abs{K}}\right) =: M_k.\notag
\end{align}

To evaluate to what extent condition (M4) holds, we assume that we have a mixed moment $\E X_{t_1}\cdots X_{t_k}$ where $\ubar{t}$ is $a$-separated by a partition $\pi$, meaning that for all $B\in\pi$, $t_B$ is $a$-connected and for all $B\neq B'\in\pi$, $d(t_B,t_{B'})> a$. We now calculate for (M4):
\begin{align}
&\E X_{t_1}\cdots X_{t_k} - \prod_{B\in\pi} \E X_{t_B}\ =\ \sum_{\ubar{s}\in T^k} c_{\ubar{t}-\ubar{s}}\E Y_{\ubar{s}} - \prod_{B\in\pi}\sum_{s_B\in T^{\abs{B}}} c_{t_B-s_B}\E Y_{s_B}\notag\\
&=\sum_{\ubar{s}\in T^k} c_{\ubar{t}-\ubar{s}}\left(\E Y_{\ubar{s}} - \prod_{B\in\pi} \E Y_{s_B}\right)\ =\ \sum_{\kappa\in\Pcal(k)_{\geq 2}} \sum_{\substack{\ubar{s}\in T^k\\ \ker(s)=\kappa}} c_{\ubar{t}-\ubar{s}}\left(\E Y_{\ubar{s}} - \prod_{B\in\pi} \E Y_{s_B}\right)\notag\\
&\underset{\abs{\ldots}}{\leq} 2 M'_k \cdot \sum_{\kappa\in\Pcal(k)_{\geq 2}\backslash\{\pi\}} \prod_{K\in\kappa} \sum_{s^K\in T} \prod_{i\in K}\abs{c_{t_i-s^K}}.\label{eq:calc}
\end{align}
In the very last product we can expect a decay whenever there are $i\neq j$ in $K$ which lie in different blocks of $\pi$, since then $t_i$ and $t_j$ are more than $a$ units apart, so the product $\abs{c_{t_i-s^K}}\abs{c_{t_j-s^K}}$ is small, since $(c_s)_s$ lies in $\ell^1(T)$. To make this more precise, assume that in $K$ we find elements of $\ell\in\{1,2,\ldots,\abs{K}\}$ different blocks of $\pi$ and we call a selection of these $i_1,\ldots,i_{\ell}$. Then the $t_{i_j}$ are pairwise $a$-disconnected and it holds
\begin{equation}
\label{eq:decayofsum}
\sum_{s^K\in T}\prod_{i\in K}\abs{c_{t_i-s^K}}\ \leq\ \ell \onenorm{(c_s)}^{\abs{K}-\ell+1}\onenorm{(c_s)_{s\in B^c_{a/2}(0)}}^{\ell-1}.
\end{equation}
We prove inequality \eqref{eq:decayofsum} by induction over $\ell$: The case is clear for $\ell=1$ by using Hölder and $\onenorm{(c_s)} \geq \norm{(c_s)}_p$ for any $p\geq 1$. For the induction step $\ell\to\ell+1$ (in which we must have $\abs{K}\geq 2$) we write $K = \{i_1,\ldots,i_{\ell+1}\}\,\dot{\cup}\, \{i_r\, |\, r\in R\}$ for a suitable set $R$ (then $\ell+1+\abs{R}=\abs{K}$) and where $i_1,\ldots,i_{\ell+1}\in K$ stem from different blocks in $\pi$ (which entails that $t_{i_1},\ldots,t_{i_{\ell+1}}$ are pairwise $a$-disconnected), and the $i_r$ are the remaining arbitrary indices. Then
\begin{align}
\sum_{s^K\in T} \prod_{i\in K}\abs{c_{t_i-s^K}}
& = \sum_{\substack{s^K\in T \\ t_{i_1}-s^K \in B_{a/2}(0)}}\abs{c_{t_{i_1}-s^K}}	\cdot\prod_{m=2}^{\ell+1}\abs{c_{t_{i_m}-s^K}} \one_{B^c_{a/2}(0)}(t_{i_m}-s^K)\cdot \prod_{r\in R}\abs{c_{t_r-s^K}}\notag\\
&\qquad + \sum_{s^K\in T}\underbrace{\abs{c_{t_{i_1}-s^K}} \one_{B_{a/2}^c(0)}(t_{i_1}-s^K)}_{\leq \onenorm{(c_s)_{s\in B_{a/2}^c(0)}}}	\cdot\prod_{m=2}^{\ell+1}\abs{c_{t_{i_m}-s^K}} \cdot \prod_{r\in R}\abs{c_{t_r-s^K}}\notag\\
&\leq \norm{(c_s)}_{\abs{K}} \cdot \norm{(c_s)_{s\in B^c_{a/2}(0)}}_{\abs{K}}^{\ell} \cdot \norm{(c_s)}_{\abs{K}}^{\abs{K}-(\ell+1)}\notag\\
&\qquad + \onenorm{(c_s)_{s\in B_{a/2}^c(0)}} \cdot \ell \cdot \norm{(c_s)_{s\in B_{a/2}^c(0)}}^{\ell-1}_{\abs{K}-1}\norm{(c_s)}_{\abs{K}-1}^{\abs{K}-1-\ell+1}\label{eq:induction}\\
&\leq (\ell+1)\onenorm{(c_s)}^{\abs{K}-\ell}\onenorm{(c_s)_{s\in B^c_{a/2}(0)}}^{\ell},\notag
\end{align}
where we used the induction hypothesis in the term \eqref{eq:induction}. This proves \eqref{eq:decayofsum}. It follows that for all $K\in\kappa$, the quantity of interest is
\[
\pi(K)\defeq \#\{B\in\pi\,|\, B\cap K\neq \emptyset\},
\]
the number of blocks of $\pi$ which intersect with $K$, which takes the role of $\ell$ in the analysis just conducted.
It then holds with \eqref{eq:decayofsum} that for any $\kappa\in\Pcal(k)_{\geq 2}\backslash\{\pi\}$,
\begin{align*}
\prod_{K\in\kappa} \sum_{s^K\in T} \prod_{i\in K}\abs{c_{t_i-s^K}} 
& \leq k^k \onenorm{(c_s)_s}^{\sum_{K\in\kappa}(\abs{K}-\pi(K)+1)}\onenorm{(c_s)_{s\in B^c_{a/2}(0)}}^{\sum_{K\in\kappa}(\pi(K)-1)}\\
& \leq k^k \onenorm{(c_s)_s}^{k}\onenorm{(c_s)_{s\in B^c_{a/2}(0)}}^{\sum_{K\in\kappa}(\pi(K)-1)}.
\end{align*}
Since we are interested in the decay we may guarantee when $\pi$ is given, we need a lower bound on $\sum_{K\in\kappa}(\pi(K)-1)$, where $\kappa\in\Pcal(k)_{\geq 2}\backslash\{\pi\}$. If $L(\pi)$ is a lower bound, we immediately obtain by the analysis in \eqref{eq:calc} that
\begin{equation}
\label{eq:MAinftyLowerBound}
\bigabs{\E X_{t_1}\cdots X_{t_k} - \prod_{B\in\pi} \E X_{t_B}}\leq 2M'_k\#\Pcal(k) k^k \onenorm{(c_s)_s}^{k} \onenorm{(c_s)_{s\in B^c_{a/2}(0)}}^{L(\pi)}.
\end{equation}
It remains to establish a lower bound $L(\pi)$ for $\sum_{K\in\kappa}(\pi(K)-1)$. First, we note that for each $K\in\kappa, \pi(K)\geq 1$. But since $\kappa\neq\pi$, there must be at least one block $K$ of $\kappa$ which intersects with two blocks of $\pi$, so then $\pi(K)\geq 2$. Therefore, a generally valid lower bound is $L(\pi)=1$, which entails at least one factor of decay $\onenorm{(c_s)_{s\in B^c_{a/2}(0)}}$. For example, if $\pi$ is a pair partition, $\pi=\{\{1,2\},\{3,4\},\{5,6\}\}$, say, and $\kappa=\{\{1,2,3,4\},\{5,6\}\}$, then $\sum_{K\in\kappa}(\pi(K)-1)=1$, so the minimal decay rate we just identified is actually assumed. On the other hand, if $\pi$ has $\ell$ single blocks, then it follows that $\sum_{K\in\kappa}(\pi(K)-1)\geq \ell/2$, since in the worst case, the singles of $\pi$ are grouped into pairs in $\kappa$, and the remaining blocks stay the same. As an example, consider $\pi=\{\{1\},\{2\},\{3\},\{4\},\{5,6\}\}$ and $\kappa=\{\{1,2\},\{3,4\},\{5,6\}\}$, which yields $\sum_{K\in\kappa}(\pi(K)-1)=2$, whereas if $\pi=\{\{1\},\{2\},\{3\},\{4\},\{5,6\}\}$ and $\kappa=\{\{1,2,3,4\},\{5,6\}\}$, we find $\sum_{K\in\kappa}(\pi(K)-1) =3$. Considering the preceding discussion and \eqref{eq:MAinftyLowerBound}, we find that condition (M4) holds with 
\begin{equation}
\label{eq:MAinftyM4}
C_k \defeq 2M'_k\#\Pcal(k) k^k \onenorm{(c_s)_s}^{k} \qquad \text{and}\qquad \gamma(a)\defeq \onenorm{(c_s)_{s\in B^c_{a/2}(0)}}.
\end{equation}

\begin{corollary}
\label{cor:MAInftyN}
Let $(c_s)_{s\in\Z}\in\ell^1(\Z)$ be arbitrary such that 
\begin{equation}
\label{eq:csDecayMAInftyN}
\sum_{s\in\Z}\abs{s}^k\abs{c_s} < \infty
\end{equation}
for all $k\in\N$. Further, let $(Y_t)_{t\in\Z}$ be a family of independent and standardized random variables, and define for all $t\in\N$:
\[
X_t \defeq \sum_{s\in\Z} c_s Y_{t-s}.
\]
Then
\begin{equation*}
\label{eq:CLTMAInftyN}
\frac{1}{\sqrt{n}}\sum_{t=1}^n X_t\	\xrightarrow[n\to\infty]{} \ \Ncal(0,\sigma^2) \qquad \text{weakly,}
\end{equation*}
where $\Ncal(0,\sigma^2)\defeq \delta_0$ if $\sigma^2 = 0$ and where
\begin{equation}
\label{eq:sigmalimitMAInftyN}
\sigma^2\ \defeq\  \lim_{n\to\infty}\frac{1}{n}\sum_{t,t'=1}^n\sum_{s\in\Z}c_{t-s}c_{t'-s}	\qquad \text{exists in $\R_+$.}
\end{equation}
\end{corollary}
\begin{proof}
We must check if in this setup, conditions (M1) through (M4) and (S1) through (S4) hold. The moment conditions (M1), (M2) and (M3) hold with the findings in \eqref{eq:MAinftyFirstMoments}, and by \eqref{eq:MAinftyM4}, (M4) holds with $C_k = 2M'_k\#\Pcal(k) k^k \onenorm{(c_s)_s}^{k}$ and $\gamma(a)=\onenorm{(c_s)_{\abs{s} > a/2}}$, which converges to $0$ as $a\to\infty$. For (S1), note that $\lim_{a\to\infty}a^k\gamma(a)=0$ for all $k\in\N$ since
\[
a^k \sum_{\abs{s}>a/2} \abs{c_s} < \sum_{\abs{s}>a/2}2^k\abs{s}^k\abs{c_s} \xrightarrow[a\to\infty]{} 0,
\]
since $(\abs{s}^k\abs{c_s})_{s\in\Z}$ is summable by \eqref{eq:csDecayMAInftyN}. (S2), (S3) hold trivially, and for (S4) we can choose $a_n\defeq\lfloor\sqrt[\ell]{n}\rfloor$. The formula for $\sigma^2$ in \eqref{eq:sigmalimitMAInftyN} follows from \eqref{eq:MAinftyFirstMoments} and \eqref{eq:sigmalimit}.
\end{proof}

\begin{corollary}
\label{cor:MAInftyZd}
Let $(c_s)_{s\in\Z^d}\in\ell^1(\Z^d)$ be arbitrary such that 
\begin{equation}
\label{eq:csDecayMAInftyZd}
\sum_{s\in\Z^d}\supnorm{s}^k\abs{c_s} < \infty
\end{equation}
for all $k\in\N$. Further, let $(Y_t)_{t\in\Z^d}$ be a family of independent and standardized random variables, and define for all $t\in\N$:
\[
X_t = \sum_{s\in\Z^d} c_s Y_{t-s}.
\]
Then with $B_n(0)=\{z\in\Z^d\,|\,\supnorm{z}\leq n\}$,
\begin{equation*}
\label{eq:CLTCor3}
\frac{1}{\sqrt{(2n+1)^d}}\sum_{t\in B_n(0)} X_t\	\xrightarrow[n\to\infty]{} \ \Ncal(0,\sigma^2) \qquad \text{weakly,}
\end{equation*}
where $\Ncal(0,\sigma^2)\defeq \delta_0$ if $\sigma^2 = 0$ and where
\begin{equation}
\label{eq:sigmalimitCor3}
\sigma^2\ \defeq\  \lim_{n\to\infty}\frac{1}{(2n+1)^d}\sum_{t,t'\in B_n(0)}\sum_{s\in\Z^d}c_{t-s}c_{t'-s}	\qquad \text{exists in $\R_+$.}
\end{equation}
\end{corollary}
\begin{proof}
The proof is concluded analogously to the proof of Corollary~\ref{cor:MAInftyN}, with the following remarks: For (M4), the decay function now takes the form $\gamma(a) = \onenorm{(c_s)_{s\in B^c_{a/2}(0)}}$. Then for (S1), we observe that $\lim_{a\to\infty}a^k\gamma(a)=0$ for all $k\in\N$ (and thus $\lim_{a\to\infty}(2a+1)^{dk}\gamma(a)=0$ for all $k\in\N$) since
\[
a^k \sum_{\supnorm{s}>a/2} \abs{c_s} < \sum_{\supnorm{s}>a/2}2^k\supnorm{s}^k\abs{c_s} \xrightarrow[a\to\infty]{} 0,
\]
since $(\supnorm{s}^k\abs{c_s})_{s\in\Z^d}$ is summable by \eqref{eq:csDecayMAInftyZd}. For (S4), we may choose $a_n\defeq \lfloor \sqrt[\ell]{n}\rfloor$ again, since
\[
\frac{2\lfloor\sqrt[\ell]{n}\rfloor +1}{\sqrt[\ell]{2n+1}} \xrightarrow[n\to\infty]{} \frac{2}{\sqrt[\ell]{2}}.
\] 

\end{proof}

Having studied processes and random fields of the MA($\infty$) type, we will now introduce autoregressive moving average (ARMA($p,q$)) processes, the theory about which will yield that these can also be treated by our framework. To this end, let $(Y_t)_{t\in\Z}$ be independent with $\E Y_t=0$ and $\V Y_t =1$ and let $p,q\in\N, a_1,\ldots,a_p, b_1,\ldots,b_q\in\R$, so that
\begin{equation}
\label{eq:ARMA}
\forall\, t\in\Z: X_t +  \sum_{j=1}^p a_j X_{t-j} = Y_t + \sum_{j=1}^q b_jY_{t-j}.
\end{equation}
Define the polynomials in $z\in\C$:
\[
A(z)\defeq \sum_{j=0}^pa_jz^j \qquad \text{and} \qquad B(z)\defeq \sum_{j=0}^q b_jz^j
\]
and let $A_0(z)$ be the \emph{reduced polynomial} that is obtained from eliminating linear factors from $A$ which pertain to common roots of $A$ and $B$ and requiring $A_0(0)=1$ (if $A$ and $B$ do not have a common root, then $A_0=A$.) It can be shown (Theorem 7.4 in \cite{Kreiss:Neuhaus:2006}) that a process $(X_t)_t$ that satisfies \eqref{eq:ARMA} exists if and only if the polynomial $A_0$ has no root $z_0\in\C$ with $\abs{z_0}=1$. In this case, it further follows that $(X_t)$ has an MA($\infty$) representation of the form $X_t=\sum_{z\in\Z}c_zY_{t-z}$, where the sequence $(c_z)_z$ is real-valued and decays exponentially fast as $\abs{z}\to\infty$, cf.\ Theorem 7.7 and Remark 7.8 in \cite{Kreiss:Neuhaus:2006}. Therefore, Corollary~\ref{cor:MAInftyN} can be applied to $(X_t)$, yielding a CLT for these general classes of ARMA($p,q$) models.

\section{Mixing Random Variables}
\label{sec:Mixing}
Let $(X_n)_{n\in\N}$ be a sequence of real-valued random variables on some probability space $(\Omega,\Acal,\Prob)$. Define for all subsets $I\subseteq\N$ the following sub-$\sigma$-algebra:
\[
\Acal_I\defeq \sigma\left(X_i\, |\, \in I\right).
\]
Further, define for all $I,J\subseteq\N$ their \emph{distance} as
\[
\dist(I,J) \defeq \inf\{\abs{i-j}: i\in I, j\in J\}
\]
and interpret the statement $I<J$ as $\sup I < \inf J$.
Then $(X_n)_n$ is called $\alpha$-mixing, if there is a sequence $(\alpha_d)_d$ in $\R_+$ with $\lim_d\alpha_d=0$ such that for all $d\in\N$ and all $I<J\subseteq \N$ with $\dist(I,J)\geq d$:
\[
\forall\,  A\in\Acal_I, B\in\Acal_J: \abs{\Prob(A)\Prob(B) - \Prob(A\cap B)} \leq \alpha_d.
\]
Similarly, $(X_n)_n$ is called $\varphi$-mixing, if there is a sequence $(\varphi_d)_d$ in $\R_+$ with $\lim_d\varphi_d=0$ such that for all $d\in\N$ and all $I<J\subseteq \N$ with $\dist(I,J)\geq d$:
\[
\forall A\in\Acal_I, B\in\Acal_J: \abs{\Prob(A)\Prob(B) - \Prob(A\cap B)} \leq \Prob(A)\varphi_d.
\]
A $\varphi$-mixing sequence is obviously also $\alpha$-mixing with $\alpha_d\defeq\varphi_d$ for all $d\in\N$, but the converse need not be the case. There are even more mixing concepts, such as $\rho$- or $\beta$-mixing, all of which are stronger than $\alpha$-mixing.  More precise connections between different mixing types can be found in \cite{Bradley:2005}. For textbook treatments of $\alpha$-mixing resp.\ $\varphi$-mixing, see \cite{Billingsley:1995} resp.\ \cite{Billingsley:1968}. The concept of $\alpha$-mixing random variables was introduced by Rosenblatt in 1956 to derive a CLT for stationary sequences, see \cite{Rosenblatt:1956}. Although as mentioned, there are plenty of stronger mixing conditions, $\alpha$-mixing is already referred to as "the strong mixing condition". This has to do with the fact that there are also substantially weaker mixing concepts in ergodic theory called "mixing" and "ergodic". The connection is that $\alpha$-mixing implies mixing, which implies ergodic, see \cite{Klenke:2020} for details.

It is interesting to note that except for trivial processes such as i.i.d.\ or $m$-dependent processes,  MA($\infty$) processes are in general "highly non-$\alpha$-mixing":
\begin{example}
Let $(Y_k)_{k\in\N_0}$ be uniformly distributed on $\{0,\ldots,9\}$ and set 
\[
\forall\, n\in\N_0: X_n \defeq \sum_{k=0}^{\infty}\frac{1}{10^k}Y_{n+k}.
\]
In words, $X_n$ is the random decimal number
 $Y_{n} .  Y_{n+1}Y_{n+2}\ldots$. 
Then for any $k\in\N, \sigma(X_n)=\sigma(Y_{n+k}\,|\, k\geq 0)$. Therefore, $(X_n)_n$ does not have a chance to be $\alpha$-mixing. To see this, let $d\in\N$ be arbitrary and consider the events $A\in\sigma(X_0,\ldots,X_k)$ and $B\in\sigma(X_{k+d},\ldots,X_{k+d+\ell})$ with $A=\{\text{the $(k+d)$-th decimal place of $X_0$ equals } 5\}=\{Y_{k+d+1}=5\}$ and  $B=\{\text{the first decimal place of $X_{k+d}$ equals } 5\}=\{Y_{k+d+1}=5\}$, then
\[
\Prob(A\cap B)-\Prob(A)\Prob(B) = \frac{1}{10}-\frac{1}{100} = \frac{9}{100},
\]
and such events can be constructed for any $d\in\N$.
\end{example}
Although $\alpha$-mixing processes are generally very different from MA($\infty$) processes, it is readily seen that these fit within the framework Theorem~\ref{thm:CLT}. We need to establish the following lemma first.

\begin{lemma}
\label{lem:mixingboundstwo}
Let $Y_1,\ldots,Y_k$, $k\geq 2$ be random variables such that for all $\ell\in\oneto{k}$, $\E\abs{Y_{\ell}}^{4(k-1)}\leq M_{4(k-1)}$ for a constant $M_{4(k-1)}\in[1,\infty)$ and $Y_{\ell}$ is $\Acal_{I_{\ell}}$-measurable, where $I_{\ell} \subseteq\N$ and $I_1< I_2 < \ldots < I_k$. Define $d_{\ell,\ell+1}\defeq \dist(I_{\ell},I_{\ell+1})$ for all $\ell\in\oneto{k-1}$. Then
\[
\abs{\E Y_1\cdots Y_k\ - \ \E Y_1\cdots \E Y_k} \leq 24 M_{4(k-1)} \sum_{\ell=1}^{k-1}\sqrt{\alpha_{d_{\ell,\ell+1}}}.
\]
\end{lemma}
\begin{proof}
Using the inequality (27.24) in \cite{Billingsley:1995} in the second step, we calculate:
\begin{align*}
&\abs{\E Y_1\cdots Y_k\ - \ \E Y_1\cdots \E Y_k}\\
&\leq \sum_{\ell=1}^{k-1} \abs{\E Y_1\cdots Y_{\ell} (Y_{\ell+1}\E Y_{\ell+2}\cdots \E Y_{k})\quad -\quad \E Y_1\cdots Y_{\ell} (\E Y_{\ell+1}\cdots \E Y_{k})}\\
&\leq \sum_{\ell=1}^{k-1} 8 \left(1 + \E\left[ (Y_1\cdots Y_{\ell})^4\right] + \E\left[(Y_{\ell+1}\E Y_{\ell+2}\cdots \E Y_{k})^4\right]\right)\sqrt{\alpha_{d_{\ell,\ell+1}}}\\
&\leq \sum_{\ell=1}^{k-1} 8 \left(M_{4(k-1)} + \max_{i\in\oneto{\ell}}\E Y_i^{4\ell} + \max_{j\in\{\ell+1,\ldots,k\}}\E Y_{j}^{4(k-\ell)} \right)\sqrt{\alpha_{d_{\ell,\ell+1}}}\\
&\leq \sum_{\ell=1}^{k-1} 8 \left(M_{4(k-1)} + \max_{i\in\oneto{\ell}}\left(\E Y_i^{4(k-1)}\right)^{\frac{\ell}{k-1}} + \max_{j\in\{\ell+1,\ldots,k\}}\left(\E Y_{j}^{4(k-1)}\right)^{\frac{k-\ell}{k-1}}\right)\sqrt{\alpha_{d_{\ell,\ell+1}}},
\end{align*}
from which the statement immediately follows.
\end{proof}

\begin{corollary}
\label{cor:Mixing}
Let $(X_n)_n$ be a sequence of $\alpha$-mixing random variables with respect to a sequence $(\alpha_d)_{d\in\N}$ with $\lim_{d\to\infty}d^k\alpha_d=0$ for all $k\in\N$. If $(X_n)_n$ satisfies the moment conditions $(M1)$, $(M2)$ and $(M3)$, then
\begin{equation}
\sigma^2\defeq \lim_{n\to\infty}\frac{1}{n}\sum_{s,t=1}^n\E X_s X_t	\qquad \text{exists in $\R_+$}
\end{equation}
and
\begin{equation}
\frac{1}{\sqrt{n}}\sum_{t=1}^n X_t	\xrightarrow[n\to\infty]{} \Ncal(0,\sigma^2) \qquad \text{weakly,}
\end{equation}
where  $\Ncal(0,\sigma^2)\defeq \delta_0$ if $\sigma^2 = 0$.

\end{corollary}

\begin{remark}
When assuming stationarity (which we do not assume at any place in this paper), the statement of Corollary~\ref{cor:Mixing} is well-known in a much stronger version, where instead of the superpolynomial decay of $\alpha_d$, only a polynomial decay is assumed, Theorem 1.7 in \cite{Ibragimov:1962}. The strength of our Theorem~\ref{thm:CLT} thus lies in its universal applicability rather than its tightness in special cases.
\end{remark}
\begin{proof}[Proof of Corollary~\ref{cor:Mixing}]
We need to verify conditions (M1) through (M4) and (S1) through (S4) so that we can apply Theorem~\ref{thm:CLT}. Per assumption in the statement of Corollary~\ref{cor:Mixing}, (M1) through (M3) are satisfied. It is also immediate that (S2) through (S4) are satisfied. It remains to show (M4) and (S1). For (M4), let $\ubar{t}\in\oneto{n}^k$ and $d\in\oneto{n}$ be arbitrary, then 
\begin{equation*}
	\bigabs{\E X_{t_1}\cdots X_{t_k} - \prod_{B\in\pi(\ubar{t} \, | \, d)}\E X_{t_B}} \leq C_k\gamma (d)
\end{equation*}
where $\gamma(d)\defeq\sqrt{\alpha_d}$ and $C_k\defeq 24(k-1)M_{4(k-1)k}^{4(k-1)k}$, which follows with Lemma~\ref{lem:mixingboundstwo}, since for any $B\in\pi\defeq\pi(\ubar{t} \, | \, d)$ we obtain by a calculation as in \eqref{eq:genHolder} below that
\[
\E\abs{X_{t_B}}^{4(\#\pi-1)} = \E \prod_{i\in B}\abs{X_{t_i}}^{4(\#\pi-1)} \leq \max_{i\in B} \E\abs{X_{t_i}}^{4(\#\pi-1)\abs{B}} \leq M_{4(k-1)k}^{4(k-1)k}.
\]
Since by assumption, $(\alpha_d)_{d\in\N}$ satisfies $\lim_{d\to\infty}d^k\alpha_d=0$ for all $k\in\N$, we observe that (S1) is satisfied.
\end{proof}

\section{Proof of Theorem~\ref{thm:CLT}}
\label{sec:proof}

Inspecting \eqref{eq:sigmalimit}, since
\[
\frac{1}{\abs{T_n}}\sum_{s,t\in T_n}\E X_sX_t = \E\left(\frac{1}{\sqrt{\abs{T_n}}}\sum_{t\in T_n}X_t\right)^2\geq 0,
\]
the limit is non-negative in case it exists. To prove existence, note that
\begin{equation}
\label{eq:sumsplit}
\frac{1}{\abs{T_n}}\sum_{s,t\in T_n}\E X_sX_t = \frac{1}{\abs{T_n}}\sum_{t\in T_n}\E X_t^2 + \frac{1}{\abs{T_n}}\sum_{s\neq t \in T_n}\E X_sX_t,
\end{equation}
and by assumption (M2), the first sum on the r.h.s.\  of \eqref{eq:sumsplit} converges. To examine the second sum, we calculate, where $S_a(t)\subseteq T_n$ refers to the sphere of radius $a$ around $t$ in $T_n$ and the constant $C_2$ stems from condition (M4):
\begin{align*}
&\frac{1}{\abs{T_n}}\sum_{s\neq t\in T_n}\abs{\E X_sX_t}\ =\ \frac{1}{\abs{T_n}} \sum_{a=1}^{\diam(T_n)}\sum_{\substack{s,t\in T_n\\ d(s,t)=a}}\abs{\E X_s X_t}\\
& \leq\ \frac{C_2}{\abs{T_n}}\sum_{a=1}^{\diam(T_n)} \abs{\{(s,t)\in T_n^2\, |\, d(s,t)=a\}}\cdot\gamma(a)
\ =\ C_2\sum_{a=1}^{\diam(T_n)} \left(\frac{1}{\abs{T_n}}\sum_{t\in T_n} \abs{S_a(t)}\right)\cdot\gamma(a)\\
& \leq \ C_2\sum_{a=1}^{\diam(T_n)}b_n(a)\gamma(a)  \leq C_2\sum_{a=1}^{\diam(T_n)} C \abs{T_a}^c \gamma(a)
\end{align*}
which is summable as $n\to\infty$, since for all $a$ large enough, $\abs{T_a}^{c+2}\gamma(a)\leq 1$, such that $\abs{T_a}\gamma(a) \leq 1/\abs{T_a}^2 \leq 1/a^2$. Here, we used condition (S1), (S2) and (S3). This shows \eqref{eq:sigmalimit}. To prove the weak convergence statement in \eqref{eq:CLT}, we employ the method of moments. To this end we need to show that for all $k\in\N$:
\begin{equation}
\label{eq:goal}
\E\left(\frac{1}{\sqrt{\abs{T_n}}}\sum_{t
\in T_n}X_t\right)^k = \frac{1}{\abs{T_n}^{\frac{k}{2}}} \sum_{\ubar{t}\in T_n^k}\E X_{t_1}X_{t_2}\cdots X_{t_{k}} \ \xrightarrow[n\to\infty]{}\ (k-1)!!\, \sigma^k\, \one_{2\N}(k).
\end{equation}

To this end, we partition the set $T_n^k$ according to the \emph{$a$-connectedness} of the tuples which we introduced at the beginning of the paper. We define for any $a\in\N_0$ and $\pi\in\Pcal(k)$:
\[
T_n^k(\pi|a) \defeq \left\{\ubar{t}\in T_n^k\,|\, \forall\, B\in\pi: t_B\text{ is $a$-connected and } \forall\, B\neq B'\in\pi: d(t_B,t_{B'})>a\right\}.
\]
The following lemma helps to obtain bounds on the sets $T_n^k(\pi|a)$:
\begin{lemma}
\label{lem:combinatorics}
For all $\pi\in\Pcal(k)$ and $a\in\N_0$: 
\[
\abs{T_n^k(\pi|a)}\ \leq\ \prod_{B\in\pi}\abs{T_n} b_n(a)^{\abs{B}-1} \abs{B}!\ \leq\ k!\abs{T_n}^{\abs{\pi}}b_n(a)^{k-\abs{\pi}}.
\]
\end{lemma}
\begin{proof}
Given $\pi\in\Pcal(k)$ and $a\in\N_0$, we construct an upper bound on the number of possibilities we have to construct a $\ubar{t}\in	 T_n^k(\pi|a)$? First, we order the blocks in $\pi$ in increasing order, that is, we write $\pi=(B_1,\ldots,B_{\ell})$, where $\ell$ is the number of blocks in $\pi$ and $1\in B_1$, $\min(\oneto{k}\backslash B_1) \in B_2$, $\min(\oneto{k}\backslash (B_1\cup B_2)) \in B_3$ and so forth. Then we proceed successively for each block $B_i$ as follows, starting with block $B_1$: We choose $t_{\min B_i}\in T_n$, having at most $\abs{T_n}$ choices. Then, continuing in increasing order through the indices of the block $B_i$, we choose the next $t_j$, $j\in B_i$, in the $a$-proximity of the previous one, admitting at most $b_n(a)$ choices. At last, we permute the tuple $(t_j\,|\, j\in B_i)$, for which we have $\abs{B}!$ choices. This shows the first inequality, and the second inequality is trivial.
\end{proof}

We proceed to show \eqref{eq:goal}. To this end, let $k\in\N$ be fixed throughout the remainder of the proof. For each $\pi\in\Pcal(k)$, we analyze the limit of
\[
\frac{1}{\abs{T_n}^{\frac{k}{2}}} \sum_{\ubar{t}\in T_n^k(\pi|a_n)}\E X_{t_1}X_{t_2}\cdots X_{t_{k}}
\]
for $n\to\infty$, where $a_n$ is a dynamic distance parameter with $a_n\nearrow\infty$, more precisely, we choose $(a_n)_n$ as in condition (S4) with $\ell\defeq\lceil c(k+1)\rceil +1$, where the constant $c$ stems from condition (S2).

\noindent\underline{Case 1: $\pi$ has at least one single block.}\newline
In this case, we note that by condition (M4), 
\[
\forall\,\ubar{t}\in T_n^k(\pi|a_n):\ \abs{\E X_{t_1}\cdots X_{t_k}}\leq C_k\gamma(a_n).
\]
Therefore, using the trivial bound $\abs{T_n^k(\pi|a_n)}\leq \abs{T_n}^k$  and condition (S1) we find
\[
\frac{1}{\abs{T_n}^{\frac{k}{2}}} \sum_{\ubar{t}\in T_n^k(\pi|a_n)}\abs{\E X_{t_1}\cdots X_{t_k}}\leq \abs{T_n}^{\frac{k}{2}}\cdot C_k\gamma(a_n) = \frac{\abs{T_n}^{\frac{k}{2}}}{\abs{T_{a_n}}^{\frac{\ell k}{2}}}\abs{T_{a_n}}^{\frac{\ell k}{2}}
C_k\gamma(a_n)\xrightarrow[n\to\infty]{} 0,
\]
where we used conditions (S1) and (S4).

\noindent\underline{Case 2: $\forall\, B\in\pi:\ \abs{B}\geq 2$ and $\exists\, B\in\pi$: $\abs{B}\geq 3$.}\newline
In this case, we find by Lemma~\ref{lem:combinatorics} that
\[
\abs{T_n^k(\pi|a_n)} \leq k! \abs{T_n}^{\frac{k-1}{2}} b_n(a_n)^{\frac{k+1}{2}}.
\]
Now for each $\ubar{t}\in T_n^k(\pi|a_n)$ we calculate
\begin{equation}
\label{eq:genHolder}
\abs{\E X_{t_1}\cdots X_{t_k}} \leq \left(\E \abs{X_{t_1}}^k\right)^{\frac{1}{k}}\cdots\left(\E \abs{X_{t_k}}^k\right)^{\frac{1}{k}} \leq \max_{\ell=1,\ldots,k}\E \abs{X_{t_{\ell}}}^k \leq M_k,
\end{equation}
where we used condition (M3), and where we used the generalized Hölder inequality in the first step. Combining these bounds we find
\begin{equation}
\label{eq:thirdmomentdecay}
\frac{1}{\abs{T_n}^\frac{k}{2}}\sum_{\ubar{t}\in T_n^k(\pi|a_n)}\abs{\E X_{t_1}\cdots X_{t_k}} \leq M_k \cdot k! \cdot \frac{b_n(a_n)^{\frac{k+1}{2}}}{\abs{T_n}^{\frac{1}{2}}} \xrightarrow[n\to\infty]{} 0,
\end{equation}
since by the choice of $\ell$ and conditions (S2), (S3) and (S4),
\[
b_n(a_n)^{k+1} \leq C^{k+1}\abs{T_{a_n}}^{c(k+1)} \leq C^{k+1}\abs{T_{a_n}}^{\ell - 1} = o(\abs{T_n}).
\]

\noindent\underline{Case 3: $\forall\, B\in\pi:\ \abs{B}= 2$.}

 It remains to investigate the partial sum over the set $T_n^k(\pi|a_n)$ for $\pi\in\Pcal\Pcal(k)\defeq\{\pi\,|\, \pi \text{ is a pair partition of } \oneto{k}\}$. If $k$ is odd, then this set is empty, so our above analysis shows that odd moments converge to zero, that is, we established that \eqref{eq:goal} holds for odd $k$. In the following, we assume $k$ to be even and calculate
\begin{align}
&\frac{1}{\abs{T_n}^{\frac{k}{2}}} \sum_{\ubar{t}\in T_n^k(\pi|a_n)}\E X_{t_1}X_{t_2}\cdots X_{t_{k}}\notag\\
&= \frac{1}{\abs{T_n}^{\frac{k}{2}}} \sum_{\ubar{t}\in T_n^k(\pi|a_n)} \left(\E X_{\ubar{t}} - \prod_{B\in\pi}\E X_{t_B}\right)\ +\ \frac{1}{\abs{T_n}^{\frac{k}{2}}} \sum_{\ubar{t}\in T_n^k(\pi|a_n)}\prod_{B\in\pi}\E X_{t_B}.\label{eq:RestGood}
\end{align}
For the first term in \eqref{eq:RestGood} we calculate
\begin{equation}
\label{eq:squareseparation}
\bigabs{\frac{1}{\abs{T_n}^{\frac{k}{2}}} \sum_{\ubar{t}\in T_n^k(\pi|a_n)} \left(\E X_{\ubar{t}} - \prod_{B\in\pi}\E X_{t_B}\right)}\leq k!\,b_n(a_n)^{\frac{k}{2}}C_k\gamma(a_n) \xrightarrow[n\to\infty]{} 0,
\end{equation}
where we used condition (M4), Lemma~\ref{lem:combinatorics} and condition (S1).
For the second term in \eqref{eq:RestGood} we calculate
\begin{align}
\frac{1}{\abs{T_n}^{\frac{k}{2}}} \sum_{\ubar{t}\in T_n^k(\pi|a_n)}\prod_{B\in\pi}\E X_{t_B} &= \frac{1}{\abs{T_n}^{\frac{k}{2}}}\sum_{\substack{\ubar{t}\in T_n^k\\ \forall\, B\in\pi: \diam(t_B)\leq a_n}} \prod_{B\in\pi}\E X_{t_B} + o(1)\label{eq:firstequality}\\
&=  \prod_{B\in\pi}\left(\frac{1}{\abs{T_n}}\sum_{\substack{t_B\in T_n^2\\ \diam(t_B)\leq a_n}}\E X_{t_B}\right) + o(1) \notag\\
&=  \prod_{B\in\pi}\left(\frac{1}{\abs{T_n}}\sum_{t_B\in T_n^2}\E X_{t_B} \right) + o(1)	\xrightarrow[n\to\infty]{} (\sigma^2)^{\frac{k}{2}}, \label{eq:thirdequality}
\end{align}
where the equality in \eqref{eq:firstequality} follows from the fact that for any $\ubar{t}\in T_n^k$ for which all $t_B$, $B\in\pi$, are $a_n$-connected but not all of them pairwise $a$-disconnected among each other, it holds that $\pi(\ubar{t} \, | \, a_n)$ is of the category of Case 2. Therefore, \eqref{eq:firstequality} follows since for such $\ubar{t}$, $\abs{T_n^k(\pi(\ubar{t} \, | \, a_n)|a_n)} =o(\abs{T_n}^{k/2})$ by Lemma~\ref{lem:combinatorics} and $\E\abs{X_sX_t}\leq M_2$. The equality leading to \eqref{eq:thirdequality} follows with
\begin{align*}
&\bigabs{\frac{1}{\abs{T_n}}\sum_{t_B\in T_n^2}\E X_{t_B}\ -\ \frac{1}{\abs{T_n}}\sum_{\substack{t_B\in T_n^2 \\ \diam(t_B)\leq a_n}} \E X_{t_B}} \\
&\leq \frac{1}{\abs{T_n}}\sum_{\substack{s,t\in T_n^2 \\ d(s,t)>a_n}} \bigabs{\E X_sX_t} \leq \abs{T_n}C_2\gamma(a_n) = \frac{\abs{T_n}}{\abs{T_{a_n}}^{\ell}}\abs{T_{a_n}}^{\ell}C_2\gamma(a_n)\xrightarrow[n\to\infty]{} 0,
\end{align*}
where we used conditions (S1), (S4) and (M4).
Since $\pi\in\Pcal\Pcal(k)$ was arbitrary and $\#\Pcal\Pcal(k)=(k-1)!!$ we obtain
\[
\sum_{\pi\in\Pcal\Pcal(k)}\frac{1}{\abs{T_n}^{\frac{k}{2}}} \sum_{\ubar{t}\in T_n^k(\pi|a_n)}\E X_{t_1}X_{t_2}\cdots X_{t_{k}} \xrightarrow[n\to\infty]{} (k-1)!!\sigma^k
\]
which finishes the proof of Theorem~\ref{thm:CLT}.

\vspace{1cm}
\noindent\textsf{(Michael Fleermann and Werner Kirsch)\newline
FernUniversit\"at in Hagen\newline
Fakult\"at f\"ur Mathematik und Informatik\newline 
Universit\"atsstra\ss e 1\newline 
58084 Hagen}\newline
\textit{E-mail addresses:}\newline
\texttt{michael.fleermann@fernuni-hagen.de}\newline
\texttt{werner.kirsch@fernuni-hagen.de}
\vspace{1cm}

\end{document}